\documentclass[12pt]{amsart}

\usepackage[left=3cm,top=3.0cm,right=3cm,bottom=2.6cm]{geometry}

\usepackage[ansinew]{inputenc}
\usepackage{amscd,amsfonts,amsmath,amssymb,amsthm,enumerate,epsfig,float,graphics,graphicx,pst-all,yhmath}
\newcommand{\inte}{\operatorname*{int}}
\newcommand{\aff}{\operatorname*{aff}}
\newcommand{\bd}{\operatorname*{bd}}
\newcommand{\rint}{\operatorname*{rint}}

\newcommand{\Rt}{\mathbb{R}^3}
\newcommand{\Rn}{\mathbb{R}^{n}}

\newtheorem{lemma}{Lemma}

\newtheorem{theorem}{Theorem}

\newtheorem{definition}{Definition}
\newtheorem{conjecture}{Conjecture}
\newtheorem{question}{Question}

\title {Convex bodies with sections with hyperplanes of symmetry}
\author{Efr\'en Morales-Amaya}
\address{Facultad de Matem\'aticas-Acapulco,
Universidad Aut\'onoma de Guerrero, M\'exico}
\email{emoralesamaya@gmail.com}
   
\begin{document}

\maketitle
\begin{abstract}  
               Let $K\subset \Rn$ be a convex body  and let $p\in \inte K$, $n \geq 3$. The 
               point $p$ is said to be a \textit{Larman point} of $K$ if, for every hyperplane 
               $\Pi$ passing through $p$, the section $\Pi\cap K$ has a $(n-2)$-plane of 
               symmetry. If, in addition, for every hyperplane $\Pi$ passing through $p$, the 
               section $\Pi\cap K$ has a $(n-2)$-plane of symmetry which contains $p$, 
               then the point $p$ is called a revolution point. In this work we prove that if 
               for the convex body $K$, $n \geq 3$, there exists a 
               hyperplane $H$, a point $p$ such that $p$ is a Larman point of $K$ but 
               not a revolution point and, for every hyperplane $\Pi$ passing though $p$, 
               the section $\Pi \cap K$ has an $(n-2)$-plane of symmetry parallel to $H$, 
               then $K$ is an ellipsoid of revolution with an axis perpendicular to $H$. 
\end{abstract}
\section{Introduction}
In this work, we study Geometric Tomography problems in which we are given information about the symmetries of sections of a convex body $K \subset \Rn$, $n\geq 3$, and want to obtain %from them 
information about the symmetries of $K$. 
\begin{question}\label{prob3} 
What can we say about a convex body $K\subset \Rn$, $n \geq 3$, with the property that there exists a point $p\in \Rn$ such that all hyperplane sections of $K$ passing through $p$ possess a certain type of symmetry? 
\end{question}
A particularly case of Question \ref{prob3} occurs when all the sections of $K$ through the point $p$ are assumed to be centrally symmetric, but $p$ is not the center of symmetry of $K$. This problem is known as the False Centre Theorem  of Aitchison-Petty-Rogers and Larman. The conjecture was confirmed in \cite{ai-pe-ro} in the case when the false centre is an interior point of $K$. The False Centre Theorem was proven in all its generality in \cite{la}. In \cite{mm2} was given a short proof of the False Centre Theorem which used a characterization of centrally symmetric body in the plane \cite{ro2} and a characterization of the ellipsoid in terms of planar shadow boundaries (see theorem 16.14 in \cite{bu}). 

On the other hand, K. Bezdek formulated the following conjecture in which the sections have axial symmetry. 

\begin{conjecture}\cite[pg. 221]{blnp}\label{bez}
If all plane sections of a convex body $K \subset \mathbb{R}^3$ have an axis of symmetry, then $K$ is an ellipsoid or a body of revolution.
\end{conjecture}
In \cite{mo}, Montejano gave an example showing that considering only sections through a fixed point is not enough. Indeed, in $\mathbb{R}^3$, the convex hull of two perpendicular discs centered at the origin has the property that every section through the origin has an axis of symmetry. 
\begin{definition}
A point $p\in \inte K$ is said to be a \textit{Larman point} of $K$ if   for every hyperplane $\Pi$ passing through $p$ %, such that the section $\Pi\cap K$  has non-empty interior, 
the section $\Pi\cap K$ has a $(n-2)$-plane of symmetry. 
\end{definition}
\begin{definition}
Let $p \in \inte K$ be a Larman point of $K$. We call $p$ a {\it revolution point} of $K$ if for every hyperplane $\Pi$ passing through $p$ %such that the section $\Pi\cap K$  has non-empty interior, 
the section $\Pi\cap K$ has a $(n-2)$-plane of symmetry which contains $p$.
\end{definition}

 As examples of Larman and revolution points, we note that if $c$ is the centre of an ellipsoid $E\subset \Rn$, which is not a body of revolution,  then $c$ is a revolution point of $E$.
 %and there is no other revolution point of $E$. 
 Furthermore, every point $p \neq c$ in the interior of the ellipsoid is a Larman point, but not a revolution point.  On the other hand, every point on the axis of a body of revolution  is a revolution point, while every point $p$ not on the axis is a Larman point. These facts were proved in \cite{alfonseca} (Corollary 2, page 13). Another results related with the notion of revolution point of a convex body were given in \cite{gmm}.

 With this terminology, we state the following refinement of Conjecture \ref{bez}.
\begin{conjecture}\label{amaya}
Let $K\subset \Rn$ be a convex body, $n\geq 3$. Suppose that $p \in \inte K$ is a Larman  point of $K$ which is not  a revolution point of $K$. Then either $K$ is an ellipsoid or $K$ is a body of revolution.  
\end{conjecture}  

Observe that if a Larman point $p$ is also the center of symmetry of $K$, then $p$ is a revolution point of $K$. Hence, Montejano's example of the convex hull of two discs is now excluded by the assumption that the Larman point $p$ is not a revolution point. 

In \cite{alfonseca} were proved a serie of interesting results where the notions of Larman point and revolution point were considered. 
In order to state some of them, we need the following definition. We say that a line $L$ is an \textit{axis of symmetry} of $K$ if, on the one hand, all sections of  $K$ by hyperplanes orthogonal to $L$ are centrally symmetric with center at a point in $L$, and on the other hand, all sections of $K$ by hyperplanes containing  $L$ have $L$ as a line of symmetry (\textit{i.e.}, given a hyperplane $H$, for every point $x\in K \cap H$, its reflection with respect to $L$ is also in $K\cap H$). Between the most important results proved in \cite{alfonseca} we have:

\begin{itemize}
\item [(I)] \textit{If $K\subset \Rt$ is a centrally symmetric strictly convex body with centre at $o$,  $K$ has two distinct revolution points $p,q$ such that $p\not= o\not= q$ and $o \notin L(p,q)$, then $K$ is a sphere}.  

\item [(II)] \textit{If $K\subset \mathbb{R}^3$ is an origin symmetric, strictly convex body, $L$ is an axis of symmetry of $K$, $p\in (\inte K )\setminus L$ is a Larman point of $K$ and, for every plane $\Pi$ passing through $p$, the section $\Pi \cap K$ has a line of symmetry which contains the point $\Pi\cap L$, then $K$ is a body of revolution with axis $L$}.
\end{itemize}
 
Our main result is the following theorem which represent a progress in the understanding of Conjecture \ref{amaya} and it was motivated by the result (II) above:

\begin{theorem}\label{patitasricas}
              Let $K\subset \Rn$ be a convex body, let $H$ be a 
              hyperplane and let $p\in \inte K$, $n\geq 3$. Suppose that $p$
              is a Larman point but not revolution point of $K$ and $p$ is 
              not in a hyperplane of symmetry of $K$ parallel to $H$ and, furthermore,
              for every hyperplane  $\Pi$ passing through $p$, the section 
              $\Pi \cap K$ has a $(n-2)$-plane of symmetry parallel to $H$.  
              Then $K$ is an ellipsoid of revolution with an axis of revolution 
              perpendicular to $H$. 
              \end{theorem}

              If we compare the result (II) with the Theorem \ref{patitasricas}, then in the 
              Theorem \ref{patitasricas} we assume that the line $L$ is a line at the 
              infinite in $H$, however, we do not assume the strictly convexity and no one 
              (global) symmetry of $K$, i.e., we do not assume in Theorem \ref{patitasricas}
              neither the existence of a center of symmetry nor an axis of symmetry.

              Our \textit{line of attack} for the proof of the Theorem \ref{patitasricas} 
              is different 
              from the one 
              used in the Theorem (II). In fact, our proof follows, in general terms, the 
              strategy for proving the False Centre Theorem given in \cite{ai-pe-ro} 
              and \cite{la}. Let us briefly review this 
              strategy in dimension 3: (1)  In \cite{ro1}, Rogers proved that if the body  
              $K$ has a false centre, then it  must be centrally symmetric (thus we 
              conclude that, if we choose a system of coordinates with the origin at the 
              center of the body, the point $-p$ is also a false centre of the body $K$). 
              (2) It is shown that every section of $K$, given by a plane containing 
              $L(p,-p)$, is contained in shadow boundary, for some 
              direction $u $ parallel to a support plane of $K$ at a point in 
              $\bd K \cap L(p,-p)$. (3) The point 2) is used to prove that $K$ is a body 
              of revolution with axis the line $L(p,-p)$. (4) It is shown that the section 
              which determines the body of revolution is an ellipse.

              Our strategy for proving the Theorem \ref{patitasricas} in dimension 3 is as 
              follows: (1) It is 
              shown that $K$ has a plane of symmetry $H'$ parallel to $H$. (2) Since one of 
              the conditions of the Theorem \ref{patitasricas} is that the Larman point 
              $p$ does not lie in 
              a plane of symmetry of $K$ parallel to $H$, $p$ does not belong to $H'$. 
              Thus the image $q$ of $p$ under a 
              reflection with respect to $H'$ is different from $p$ and is also a Larman 
              point and, for every plane $\Pi$ containing $q$, the section $\Pi \cap K$ 
              has a line of symmetry parallel to $H$. (3) we show that there are no line 
              segments contained in $\bd K$ and parallel to $H$, (4) for every direction 
              $u$ parallel to $H$, we show that the shadow boundary of $K$ 
              corresponding to $u$ is contained in a plane perpendicular to $u$. (5) 
              using (4) we show that $K$ is a body of revolution with axis orthogonal to 
              $H$. (6) It is shown that every section of $K$ through $p$ has two 
              perpendicular lines of symmetry and is therefore centrally symmetric (it is 
               important to observe here that every section of a body of revolution has a 
               line of symmetry which intersects the axis of revolution). By 
              virtue of the False Centre Theorem, $K$ has center at $p$ or $K$ is an 
              ellipsoid, however, $p$ cannot be the center of $K$ because in that case 
              $p$ should lie in the plane of symmetry of $K$ but that would contradict our 
              choice of $p$.
\section{Basic notions and some auxiliary results}
Let $\mathbb{R}^n$ be the $n$-dimensional Euclidean space endowed with the standard inner product $\langle \cdot, \cdot \rangle : \mathbb{R}^n \times \mathbb{R}^n \to \mathbb{R}$. We choose an orthonormal system of coordinates $(x_1, \dots, x_n)$ for $\mathbb{R}^n$. Define the closed unit $n$-ball by
	\[
	B(n) = \{x \in \mathbb{R}^n : \|x\| \leq 1\}, \quad S^{n-1} = \{x \in \mathbb{R}^n : \|x\| = 1\}
	\]
	as its boundary. Let $x, y \in \mathbb{R}^n$. Denote by $L(x, y)$ the line through $x$ and $y$, and by $[x, y]$ the line segment connecting them. A {\it chord} of a convex body $K$ is any line segment $[x, y]$ in $K$ such that $a,b\in \bd K$. A {\it body} in $\mathbb{R}^n$ is a compact set which is equal to the closure of its nonempty interior.  A {\it convex body} is a body $K$ such that for every pair of points  $a,b\in K$ the line segment $[a,b]$ is contained in $K$. A convex body is {\it strictly convex} if its boundary does not contain a line segment. For each unit vector $\xi \in \mathbb{R}^n$, a chord parallel to $\xi$ of maximal length is called a {\it diametral chord} of $K$. For sets $A, B \subset \mathbb{R}^n$, let $\mathrm{aff}\{A, B\}$ be the affine hull of $A \cup B$.  
 	
	 A set $C\subset \Rn$ is \emph {symmetric} if and only 
               if $C=-C$.  Moreover, we say that $C$ is \emph{centrally symmetric} 
               if there is a symmetric translated copy $C+w$ of $C$; that is, if  $C$ and
               $-C$ are translated copies of each other.
\begin{figure}[H]
    \centering
     \includegraphics [width=.8\textwidth]{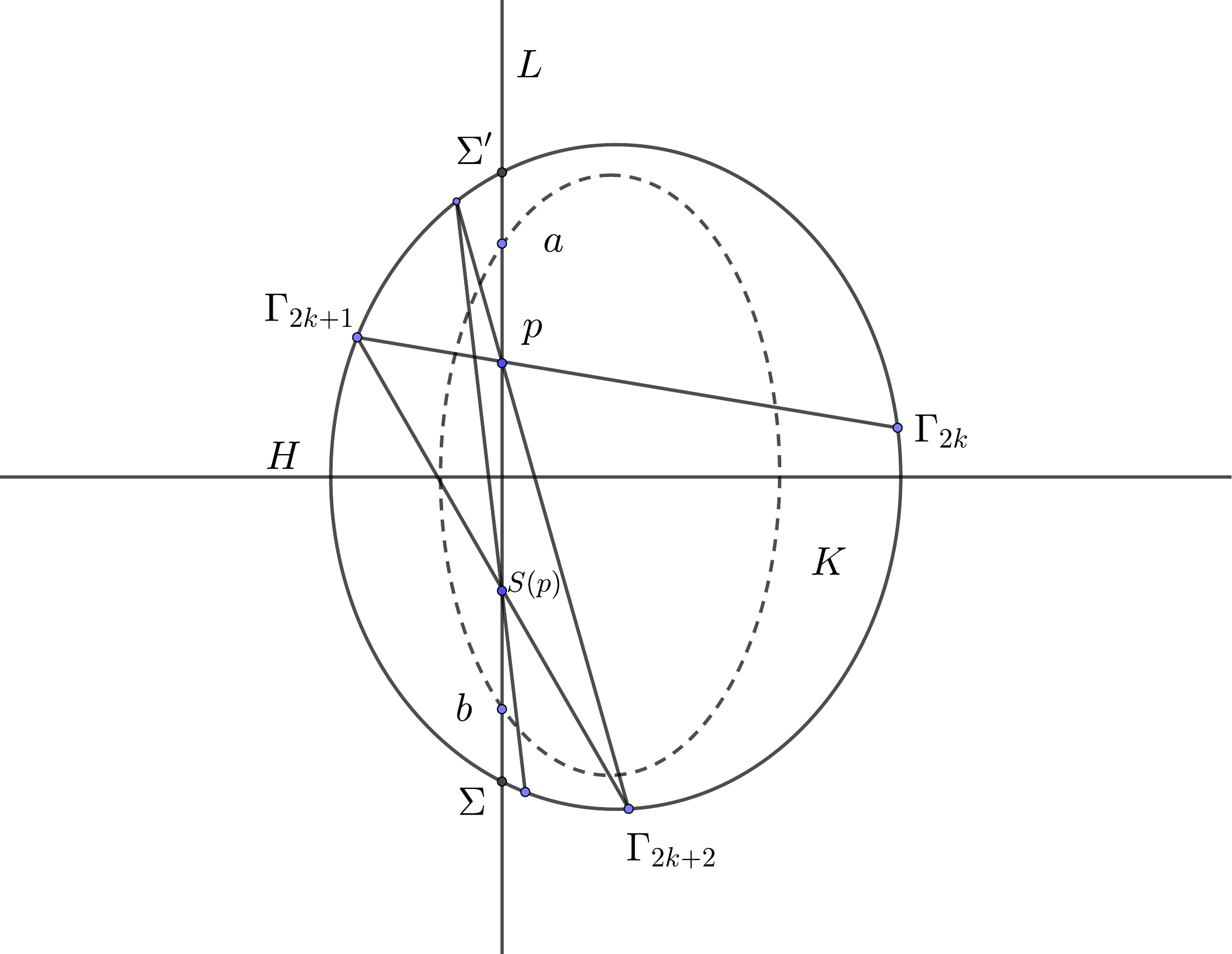}
    \caption{The orthogonal projection $K$ defined by $\Gamma$.}
    \label{sensual}
\end{figure}                
Let $H \subset \mathbb{R}^n$ be a hyperplane. A mapping $S : \mathbb{R}^n \to \mathbb{R}^n$ is a \emph{reflection} with respect to $H$ if, for every point $x \in \mathbb{R}^n$, the point $S(x)$ lies on the line orthogonal to $H$ through $x$, at equal distance from $H$, and on the opposite side of $H$ from $x$. A convex body $K \subset \mathbb{R}^n$ is said to be \emph{symmetric} with respect to $S$ if $S(K) = K$.
                 If $\Gamma$ is a $k$-plane, $1\leq  k \leq n-1$, then the 
                shadow boundary $S\partial(K,\Gamma)$ of $K$ with respect 
                to $\Gamma$ is the set in $\bd K$, defined by
\[
                S\partial(K,\Gamma)= \{\Lambda \cap K : \Lambda \textrm{ is 
                supporting } k\textrm{-plane of } K \textrm{ parallel to } 
                \Gamma\}.
\]
                 The shadow boundary $S\partial(K,\Gamma)$ of $K$ with 
                respect to $\Gamma$ is said to be a segment free if there is 
                not a line segment $I\subset \bd K$ contained in 
                $S\partial(K,\Gamma)$ parallel to $\Gamma$.
%\section{Proof of Theorem \ref{patitasricas}}
\begin{lemma}\label{planosimetria}
              The body $K$ has an hyperplane of symmetry parallel to $H$. 
\end{lemma}
\begin{proof}
             Let $L$ be a line perpendicular to $H$ and passing through $p$. For every 
             hyperplane $\Pi$ such that $L\subset \Pi$, 
             $\Pi \cap K$ has a $(n-2)$-plane of symmetry $L_{\Pi}$ parallel 
             to $H$. Then, since $L$ is orthogonal to $H$, $L_{\Pi}$ is 
             perpendicular to $L$ and it is passing through the mid-point of 
             the chord $L\cap K$. Consequently, the family of $(n-2)$-planes 
             $L_{\Pi}$, $L \subset \Pi$, determine the hyperplane of 
             symmetry of $K$. 
\end{proof}
               We denote by  $S : \mathbb{R}^n \to \mathbb{R}^n$ the reflection with 
               respect to $H$. By virtue of Lemma \ref{planosimetria} we can assume 
               that the body $K$ is symmetric with respect $H$. In this section we will 
               assume that $p$ is not in the hyperplane of symmetry $H$ of $K$. Thus 
               $p\not=S(p)$. Notice that, by virtue of the symmetry of $K$ with respect 
               to $H$, $S(p)$ is a Larman point, i.e., every hyperplane $\Gamma$ 
               passing through $S(p)$ has an $(n-2)$-plane of symmetry but, in addition, 
               we assume that such $(n-2)$-plane is parallel to $H$.

                Let $d:=n-2$. We are going to prove 
                that a shadow boundary of $K$ which correspond to 
                $d$-plane parallel to $H$ and is segment free is contained 
                in a 2-plane.
\begin{lemma}\label{yosiledoy}
                 Let $\Gamma \subset H$ be a $d$-plane. Suppose that 
                 $S\partial(K, \Gamma)$ is segment free shadow boundaries of $K$.Then 
                 $S\partial(K,\Gamma)$ is contained in a 2-plane $\Delta$, orthogonal to 
                 $\Gamma$. 
         \end{lemma}
\begin{proof}
               Let $\Gamma$ be a supporting $d$-plane of $K$ parallel to 
               $H$. Let us assume that $S\partial(K, \Gamma)$ is segment 
               free shadow boundaries of $K$. We construct a sequence of 
               $d$-planes parallel to $\Gamma$, $\{\Gamma_{s}\}^{\infty}_{s=0}$, 
               and a sequence of hyperplanes $\{\Pi_{s}\}^{\infty} _{s=0}$ in 
               the following way. We put $\Gamma_0$ equal to $\Gamma$. Let 
               $\Pi_{2k}$ be the hyperplane $\aff \{\Gamma_{2k}, p\}$, $L_{2k}$ be 
               the $d$-plane of symmetry of $\Pi_{2k}\cap K$ parallel to 
               $H$ and let $\Gamma_{2k+1}$ be the image of $\Gamma_{2k}$ 
               under the 
               symmetry with respect to $L_{2k}$, $k=0,1,2,...$ Let 
               $\Pi_{2k+1}$ be the hyperplane $\aff \{\Gamma_{2k+1},S(p)\}$, $L_{2k+1}$ 
                be the $d$-plane of symmetry of $\Pi_{2k+1}\cap K$ parallel 
                to $H$ and let $\Gamma_{2k+2}$ be the image of 
                $\Gamma_{2k+1}$ under 
                the symmetry with respect to $L_{2k+1}$, $k=0,1,2,...$ (see 
                fig. \ref{sensual}).  
                
            Let $\Pi$ be the hyperplane orthogonal to $H$, parallel to 
            $\Gamma$ and passing through $p$ and let $\Pi^{+}$ and 
            $\Pi^{-}$ be the half-spaces defined by $\Pi$. It is clear that if 
            $\Gamma_{2k}\in \Pi^{+}$, then 
            $\Gamma_{2k+1}\in \Pi^{-}$. From here, we conclude that there exist two 
            $d$-plane $\Sigma$, $\Sigma' \subset\Pi$, such that $\Sigma$, $\Sigma'$ 
            are support $d$-planes of $\Pi \cap K$ and 
\begin{eqnarray}\label{alma}
                \Gamma_{2k}\rightarrow \Sigma , k \rightarrow \infty. 
\end{eqnarray}
\begin{eqnarray}\label{liz}
                \Gamma_{2k+1}\rightarrow \Sigma' , k \rightarrow \infty. 
\end{eqnarray}
             Let $\alpha_n:=\bd K \cap \Gamma_n$, $\alpha:=\bd K \cap \Sigma'$ and 
             $\beta:\bd K \cap \Gamma$.
              Let $\Delta$ and $\Xi$ be the 2-planes orthogonal 
             to $\Gamma$ and passing through $\alpha$ and $\beta$, 
             respectively. 
             By virtue that the $d$-planes $L_{2k+1}$ are parallel to 
             $\Gamma$ it follows $\alpha_{2k+1}\in \Xi$ for all $k$. By 
             (\ref{liz}), $\alpha_{2k+1} \rightarrow \alpha$ and we conclude 
             that $\alpha \in \Xi$. Thus $\Delta= \Xi$. By the arbitrariness of 
             $\Gamma$, we derived the equality 
             $S\partial (K,\Gamma)=\Delta \cap K$.
\end{proof}
\begin{figure}[H]
    \centering
     \includegraphics [width=1.\textwidth]{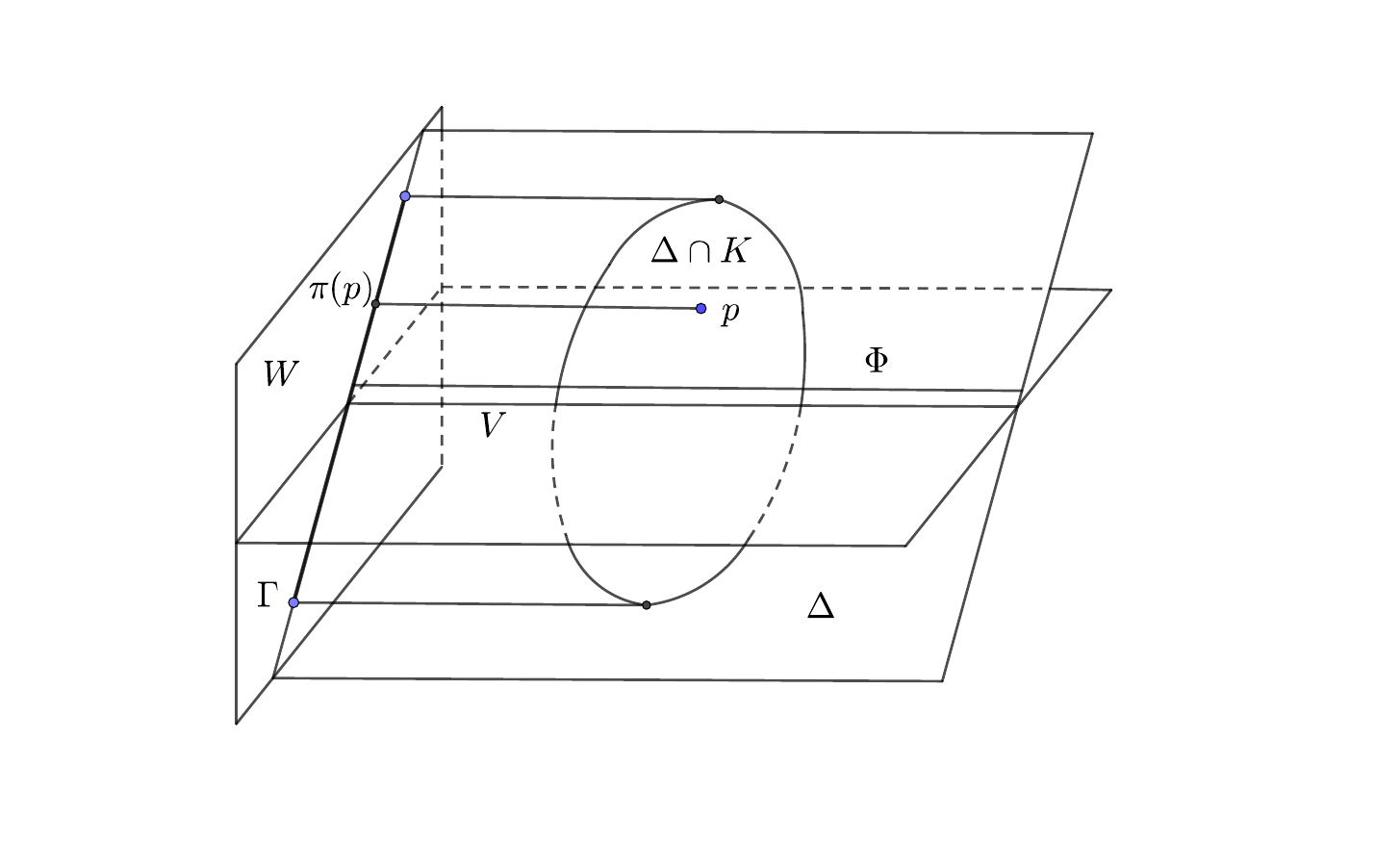}
    \caption{$\pi(K)$ is a 3-dimensional ellipsoid.}
    \label{dona}
\end{figure}
  
         %     Now suppose that $\dim(H_1\cap \bd K)=d$. Let 
            % $U,V\subset H_1$ two affine $d$-subspaces such that 
             %$U\not=V$, $H_1\cap \bd K\subset U$ and 
             %$(H_1\cap \bd K)\cap V\not= \emptyset$. We take 
             %$x\in (H_1\cap \bd K)\cap V$ and, since $\dim(U\cap V)=d-1$, 
             %there exist $y\in U\backslash V$. We choose $V$ in such a 
             %way that $(y-x)$ is not orthogonal to $V$ and $S\partial(K, V)$
              %is segment free shadow boundaries of $K$ (Such $V$ exist, 
              %otherwise, we would construct a collection $\Theta$ of line 
              %segments contained in $\bd K$, parallel to $H$ and not 
              %contained neither $H_1$ nor $H_2$ such that the set of points
               %of the unit sphere $\mathbb{S}^{n-1}$ representing the 
               %directions of those segments in $\Theta$ has no zero 
               %$d$-dimensional Hausdorff measure but this contradict 
               %Theorem 2 of \cite{laro}). We denote by $\Xi(x)$ and $\Xi(y)$ 
               %the 2-planes orthogonal to $V$ and passing through $x$ and 
               %$y$, respectively. Then $\Xi(x)\not =\Xi(y)$, i.e., 
               %$\Xi(x) \cap \Xi(y)=\emptyset,$ otherwise, $(y-x)$ would be 
               %orthogonal to $V$ which would contradict the choice of $V$. 
               %By the first part of the proof of Lemma \ref{yosiledoy}, it 
               %follows that $S\partial (K, V)\subset \Xi(x)$ and 
               %$S\partial (K, V)\subset \Xi(y)$, i.e., 
               %$\Xi(x) \cap \Xi(y)\not =\emptyset$. This contradiction shows 
               %us that it is impossible that $\dim(H_1\cap \bd K)=d$ holds. 
           
\section{Reduction of the general case of Theorem \ref{patitasricas} to dimension 3.}
              Before to show how to reduce the general case of Theorem 
              \ref{patitasricas} to dimension 3 we will prove some 
             auxiliary results
\begin{lemma}\label{suave}
               Suppose that the body $K$ has a unique diametral chord $[\alpha,\beta]$ 
               perpendicular to $H$ and, for all $d$-plane $\Gamma \subset H$, the 
               shadow boundary $S\partial(K,\Gamma)$ of $K$ corresponding to 
               $\Gamma$ is contained in a 2-plane $\Delta$, orthogonal to $\Gamma$ 
               and containing $[\alpha,\beta]$. 
               Then the body $K$ is a body of revolution with axis the line defined by 
               $[\alpha,\beta]$.
 \end{lemma}
  \begin{proof}
              In order to prove that $K$ is a body of revolution, we are going to prove that, 
              for every hyperplane $\Pi$ parallel to $H$, the section $\Pi \cap K$ is a 
              $d$-sphere and the locus of the centers lies in a line segment 
              perpendicular to $H$. 
              
              Let $\Pi$ be a plane parallel to $H$, $\Pi \cap \inte K\not= \emptyset$. In 
              order to prove to $ \Pi \cap K$ is a $d$-sphere we will use the following 
              well known characterization of the $(n-1)$-sphere: 
              
              \textit{Let \ $M\subset \Rn$ be a convex body and \ $z$ one of 
              its interior points. Let assume that for every boundary point 
              \ $x$ of \ $M$ \ is passing a supporting hyperplane orthogonal to 
              the line $L(x,z)$. Then \ $M$ is an $(n-1)$-sphere with center at \ $z$}
  
              Let $x$ be a boundary point in $\Pi \cap K$
              and $z:=\Pi \cap  [\alpha,\beta]$, $\Delta:=\aff\{[\alpha,\beta],x\}$. 
              Let $\Gamma$ be $d$-plane orthogonal to $\Delta$.
              By virtue of the hypothesis
\[
\Delta \cap K=S\partial(K, \Gamma).
 \]  
              Consequently, since $\Gamma$ is parallel to 
              $H$, there is a supporting $d$-plane of $\Pi \cap K$ 
              passing through $x$ and orthogonal to the 
              line $L(x,z)$. Thus $ \Pi \cap K$ is an $d$-sphere with center at 
              $z$. This completes the proof that $K$ is a body of revolution 
              with axis the line defined by $[\alpha, \beta]$.
\end{proof}
            Let us assume that $n>3$ and the Theorem \ref{patitasricas} 
            holds for dimension 3. Let $W$ be a hyperplane orthogonal to $H$ and let   
            $\pi: \Rn \rightarrow  W$ be the orthogonal projection onto $W$. 
\begin{lemma}\label{abejita}            
            The projection $\pi(K)$ is an $(n-1)$-ellipsoid.
\end{lemma}
\begin{proof}            
            Let $\Gamma \subset W \backslash H$ be a $(n-2)$-plane such that 
            $\pi (p) \in \Gamma$. We are going to show that there exists a 
            $(n-3)$-plane of symmetry $\Sigma$ of $\Gamma \cap \pi(K)$ parallel to 
            the $(n-2)$-plane $W\cap H$. Let $\Delta=\pi^{-1}(\Gamma)$. By 
            virtue of the hypothesis of the Theorem \ref{patitasricas}, there 
            exists an $(n-2)$-plane 
            of symmetry $\Phi$ of the section $\Delta \cap K$ parallel to 
            $H\cap \Delta$ (see fig. \ref{dona}). It 
            follows that $W$ is orthogonal to $\Phi$. Thus 
            $\Gamma \cap \pi(K) $ is symmetric with respect to 
            $\pi(\Phi)$ which is parallel to  
            $\pi(H\cap \Delta)=\pi(H)\cap \pi(\Delta)=(W\cap H)\cap \Gamma$. Hence 
            $\pi(K)$ is a $(n-1)$-dimensional ellipsoid.
\end{proof}
\begin{lemma}\label{chiquitita}
           Let $\Pi$ be an hyperplane parallel to $H$ such that 
            $\Pi\cap \inte K \not= \emptyset$. Then $\Pi\cap K$ is an 
            $(n-1)$-ellipsoid. 
\end{lemma}            
\begin{proof}   
           Let $\Gamma \subset \Pi$ be a $d$-plane and
             let $W$ be a hyperplane orthogonal to $\Pi$, $\Gamma \subset W$. 
           By virtue of the Lemma \ref{abejita}, $\pi(K)$ is a $(n-1)$-ellipsoid. 
             Thus  $\Gamma \cap \pi(K) $ is an $d$-ellipsoid. Therefore the 
             orthogonal projection $\pi(\Pi\cap K)=     \Gamma \cap \pi(K) $ of 
             $\Pi\cap K$ onto a $d$-plane is an ellipsoid. According to \cite{burton}, 
             $\Pi\cap K$ is an $(n-1)$-ellipsoid.
\end{proof} 

\subsection{Reduction of Theorem \ref{patitasricas}}              
              Let us assume that $n>3$ and the Theorem \ref{patitasricas} 
            holds for dimension 3. By Lemma \ref{abejita}, for every hyperplane $W$ orthogonal to $H$, the 
              projection $\pi(K)$ is an $(n-1)$-ellipsoid. 
              On the other hand, by Lemma \ref{chiquitita}, for hyperplane $\Pi$ parallel 
              to $H$ such that 
            $\Pi\cap \inte K \not= \emptyset$, the section $\Pi\cap K$ is an 
            $(n-1)$-ellipsoid. From this two facts it follows that $K$ is strictly convex. 
            Suppose that there exists a line segment $I$ in the boundary of $K$. If $I$ is 
            parallel to $H$, then it would exist a section of $K$, with a hyperplane parallel 
            to $H$ and containing $I$, otherwise, it would exist a projection of $K$ onto a 
            hyperplane perpendicular to $H$, and, in both cases would be a line segment 
            in the boundary of some ellipsoid which is absurd.
               
            By virtue of the strictly convexity of $K$, on the one hand, there is only one 
            diametral chord of $K$ perpendicular to $H$, which will be denote by 
            $[\alpha,\beta]$ and we choose the notation such $\alpha\in H_1$, and 
            $\beta\in H_2$ and, on the other hand, for every $d$-plane 
            $\Gamma\subset H$, 
            the shadow boundary $S\partial (K, \Gamma)$ is segment free. 
             By Lemma  \ref{yosiledoy}, for every $d$-plane $\Gamma \subset H$, 
             there exists a 2-plane $\Delta$ such that $\Delta \perp \Gamma$, 
             $[\alpha,\beta] \subset \Delta$ and 
              $S\partial(K,\Gamma) = \Delta \cap K$. 
              By Lemma \ref{suave}, the body $K$ is a body of revolution with axis the line 
              defined by $[\alpha,\beta]$.
             
             Since the body $K$ is a body of revolution, for every line 
             $\mathcal{L}\subset H$ it follows that 
             $S\partial(K, \mathcal{L}) =  \mathcal{L}^{\perp} \cap K$, where 
             $\mathcal{L}^{\perp}$ is the hyperplane perpendicular to $\mathcal{L}$. Thus 
             $\mathcal{L}^{\perp} \cap K$ is equal to the projection of $K$ parallel to 
             $\mathcal{L}$ which is an ellipsoid by Lemma \ref{abejita}. Consequently 
             $K$ is an ellipsoid of revolution with axis the line 
              defined by $[\alpha,\beta]$. 
\section{Proof of Theorem \ref{patitasricas} for $n=3$} 
              Before to prove Theorem \ref{patitasricas} we will prove some 
             auxiliary results
             \begin{lemma}\label{alexia}
              Let $H_1$ and $H_2$ be the support hyperplanes of $K$ 
             parallel to $H$. Then 
             \[
             \dim (H_i \cap \bd K)< 2, \textrm{ } \textrm{ } i=1,2.
             \] 
      \end{lemma}
             \begin{proof}
              On the contrary, let us assume that 
              $\dim (H_1 \cap \bd K)=2$. 
             Let $x\in \rint (H_1 \cap \bd K)$ and let $\Pi$ be a plane 
             containing the line $L(p,x)$. By hypothesis there exists a  
              line of symmetry $\Gamma$ of $\Pi \cap K$ parallel to $H$. 
             We denote by $\Omega'$ the image of 
             $\Omega:=\Pi \cap (H_1\cap \bd K)$ under reflection in $\Gamma$. We 
             can choose $x$ in such a way that line $L(p,x)$ is 
             not orthogonal to $H$. Thus $\Gamma \not= \Pi \cap H$. From 
             here it follows that $\Omega$ does not belong to $H_2$. Since 
             $\dim (H_1 \cap \bd K)=2$, we conclude that 
             $\dim \Omega=1$. Hence 
             $\Omega'\subset \bd K \backslash H_2$ has dimension $1$. 
             Varying $\Pi$, assuming that $L(p,x)\subset \Pi$, we construct 
             a collection $\Theta$ of line segments contained in $\bd K$, 
             parallel to $H$ and not contained neither $H_1$ nor $H_2$ 
             such that the set of points of the unit sphere 
             $\mathbb{S}^{2}$ representing the directions of those 
             segments in $\Theta$ has no zero $1$-dimensional Hausdorff 
             measure but this contradict Theorem 2 of \cite{laro}.
                 \end{proof} 
            Let $S$ be the set of directions such that $v$ is in $S$ if there exist a line 
            segment $I$ contained in $\bd K$ and parallel to $H$. 
\begin{lemma}\label{aroma}
            The set $S$ is an empty set.
\end{lemma}
 \begin{figure}[H]
    \centering
\includegraphics [width=5.5in] {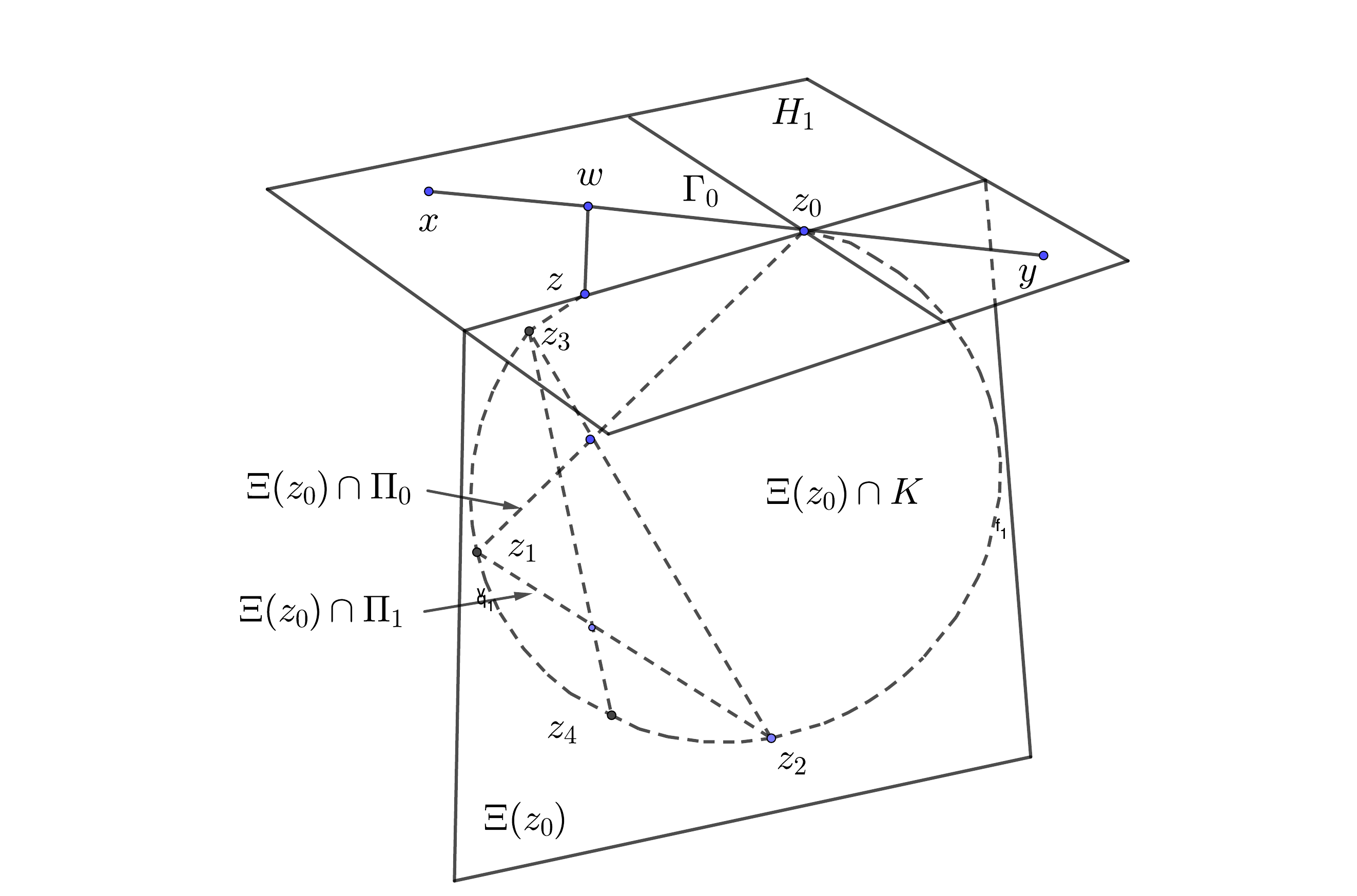}
5\caption{There is no a line segment in $\bd K$ parallel to $H$.}
    \label{super}
 \end{figure}                    
\begin{proof}            
            By Theorem 2 of \cite{laro}, the set $S$ has zero $1$-dimensional Hausdorff 
            measure. On the other 
            hand, notice that, by Lemma \ref{alexia},  
\begin{eqnarray}\label{suculenta}             
               \dim (H_i\cap  \bd K)\leq1,
\end{eqnarray}                
             $i=1,2$, there is at most a line segments in $H_i\cap  \bd K$, $i=1,2$. 
 
             Let $v\subset S$, $[x,y]\subset \bd K$ be a line segment parallel to $v$  
            and let $\Pi$ be a plane containing $[x,y]$ and parallel to $H$. 
            First we consider the case $\Pi \not=H_i$, $i=1,2$. Since $\Pi$ is not a 
            support plane of $K$ it is clear that $\Pi \cap \inte K\not=\emptyset$. Since 
            the set $S$ has zero $1$-dimensional Hausdorff measure, there exist a 
            direction $u$, $u\not=v$, such that $S\partial (K, u)$ is 
            segment free and the plane $\Delta$ which satisfies, according Lemma 
             \ref{yosiledoy}, the relations 
             \[
             u\perp \Delta, \textrm{ }\textrm{ } \textrm{ }\Delta \cap K=S\partial (K, u)  
             \]
             has the property            
    \[
       \Delta \cap \inte [x,y]\not=\emptyset.
     \]
            Hence there exist a support line $L_z$ of $\Pi \cap K$, passing through 
            $z:=\Delta \cap \inte [x,y]$, and parallel to $u$, however, since 
            $z\in \inte [x,y]$, 
            there is only one supporting line of $\Pi \cap K$ passing through $z $ and it is 
            parallel to $v$. Therefore $[x,y] \subset L_z$ and $u=v$ but this contradicts 
            our choice of $u$.  
            
            Now we suppose that $\Pi =H_i$, say $\Pi =H_1$. Let $W_0:=\aff\{p,[x,y]\}$. 
            If $W_0 $ is not perpendicular to $H$, then we repeat the argument above 
            for the line segment $[x',y']$ where $x'$ and $y'$ are images of the reflection 
            in $W_0$ with respect to the line of reflection parallel to $H$ given by the 
            hypothesis (notice that $[x',y']$ is not contained neither $H_1$ nor $H_2$). 
             
            Finally, we consider the case $W_0 \perp H$. Let 
            $z_0\in \inte [x,y]$ with $z_0\not= w:=L\cap H_1$ and let 
            $\Gamma_0 \subset H_1$ 
            be a support 
            line of $K$ passing through $z_0$ and not containing $[x,y]$, 
            let $\Pi_0:=\aff\{p,\Gamma_0\}$ and let $M$ be the plane perpendicular to 
            $H$, parallel to $\Gamma_0$ and containing the line $L$. Notice, since 
            $z_0\not= w$, $\Pi_0\not= M$. By the 
            same argument used in the 
            proof of Lemma \ref{yosiledoy}, we construct a sequence 
            $\{\Gamma_{s}\}^{\infty}_{s=0}$ of support lines of $K$ parallel to 
            $\Gamma_0$, a sequence of hyperplanes $\{\Pi_{s}\}^{\infty} _{s=0}$ and a 
            sequence of point $\{z_{s}\}^{\infty} _{s=0}$ such that 
            $z_s:=\Gamma_s \cap (\Pi_s \cap K)$. 
            Then we conclude that there exist two  points $z\in H_2$,  
            $z'\in H_1$, such that  
\begin{eqnarray}\label{perla}
                z_{2k}\rightarrow z , k \rightarrow \infty. 
\end{eqnarray}
\begin{eqnarray}\label{negra}
                z_{2k+1}\rightarrow  z' , k \rightarrow \infty. 
\end{eqnarray}
            We denote by $\Xi(z_0)$ the planes orthogonal to $\Gamma_0$ and passing 
            through $z_0$. By the construction it is clear that $z_s\in \Xi(z_0)$, for all 
            $s$, and $z\in \Xi(z_0)$ (see fig. \ref{super}). By virtue that the body $K$ is 
            closed, it follows that 
            $z\in (\Xi(z_0)\cap \bd K)$. Since 
            $z_0\not= w$, we conclude that $z\not= w$. Furthermore, since 
            $\Pi_0\not= M$, it follows that  $z\not=z_0$. Thus the line 
            segment $[w,z]$, on the one hand, is contained in $\bd K$ (by the 
            convexity of $K$) and, on the other hand, $[w,z]\not=[x,y]$. Thus 
            $\dim (H_1\cap K)=2$. However this contradicts  (\ref{suculenta}) (here 
            we assume that $w\in [x,y]$ but, otherwise, we take a point $w'\in [x,y]$, 
            $w'\not =z_0$, and we get to the same conclusion). 
            
            All this considerations drive us to conclude that  $S=\emptyset$. 
 \end{proof}
 \subsection{Proof of Theorem \ref{patitasricas}.} 
             By Lemma \ref{aroma}, there is only one diametral chord of $K$ 
             perpendicular to $H$. We denote by $[\alpha,\beta]$ the diametral chord of 
             $K$ perpendicular to $H$ and we choose the notation such $\alpha\in H_1$ 
             and $\beta\in H_2$. On the other hand, by Lemma 
             \ref{aroma}, for every line  $\Gamma\subset H$, the shadow boundary 
             $S\partial (K, \Gamma)$ is segment free. By Lemma  \ref{yosiledoy}, for 
             every line $\Gamma \subset H$, 
             there exists a plane $\Delta$ such that $\Delta \perp \Gamma$, 
             $[\alpha,\beta] \subset \Delta$ and 
              $S\partial(K,\Gamma) = \Delta \cap K$. 
              By Lemma \ref{suave}, the body $K$ is a body of revolution.  
                  
              Since $K$ is a body of revolution with axis the line defined by 
              $[\alpha, \beta]$, every section of $K$, passing through $p$, has two 
              orthogonal lines of symmetry. Let $W$ be a plane, 
              $p\in W$. There exists a plane of symmetry of $K$, say 
              $\Pi$, orthogonal to $W$ and containing 
              $[\alpha,\beta]$. Hence $\Pi \cap  W$ is a line of symmetry of 
              $K\cap  W$. On the other hand,  $\Pi \cap  W$ is 
              orthogonal to the line of symmetry of $K\cap  W$ parallel 
              to $H$, given by the hypothesis. Thus, all the section of $K$ 
              passing through $p$ are centrally symmetric and, by virtue of 
              the False Centre Theorem $K$ is an ellipsoid.

\end{document}